\documentclass{amsart}

\usepackage[english]{babel}

\usepackage[letterpaper,top=2cm,bottom=2cm,left=3cm,right=3cm,marginparwidth=1.75cm]{geometry}

\newtheorem{theorem}{Theorem}[section]
\newtheorem{lemma}[theorem]{Lemma}
\newtheorem{proposition}[theorem]{Proposition}

\newtheorem{corollary}[theorem]{Corollary}

\theoremstyle{definition}
\newtheorem{definition}[theorem]{Definition}

\theoremstyle{remark}

\usepackage{amsmath,amssymb}
\usepackage{graphicx}
\usepackage[colorlinks=true, allcolors=blue]{hyperref}
\usepackage{todonotes}
\usepackage{tikz}
\usetikzlibrary{decorations.pathmorphing}
\usepackage{ytableau}
\usepackage{verbatim}
\usepackage{diagbox}

\newcommand{\B}{\mathcal{B}}

\newcommand{\rank}{\operatorname{rank}}

\let\depth\relax
\newcommand{\Le}{\raisebox{\depth}{\scalebox{-1}[-1]{\ensuremath{\Gamma}}}}

\newcommand{\sm}{\setminus}

\usepackage{xspace}

\title{Sparse Paving Positroids}
\author{John Machacek \and George D. Nasr}
\address{Department of Mathematics and Computer Science, Hampden-Sydney College} \email{jmachacek@hsc.edu}
\address{Department of Mathematics, Augustana University} \email{george.nasr@augie.edu}

\begin{document}
\maketitle

\begin{abstract}
Using Postnikov's $\Le$-diagrams, decorated permutations, and Grassmann necklaces, we classify which positroids are sparse paving matroids. This allows us to enumerate sparse paving positroids, making connections to a known sequence involving the golden ratio and to the Lucas numbers.
\end{abstract}

\section{Introduction}
In this note, we answer the question: when is a positroid a sparse paving matroid?
In studying the totally nonnegative Grassmannian, Postnikov defined positroids~\cite{Pos} as real-representable matroids that admit a matrix representation whose maximal minors are all nonnegative.
These turn out to be equivalent to the base-sortable matroids of Blum~\cite{Blum}, as shown in~\cite[Corollary 9.5]{LamPos}.

Sparse paving matroids form a special subclass of paving matroids.
Paving matroids have long been believed to constitute a “large class’’ among all matroids~\cite{CrapoRota}.
It is conjectured that asymptotically almost all matroids are paving matroids, which would imply that asymptotically almost all matroids are sparse paving matroids~\cite[Conjecture 1.6]{paving}.
We characterize precisely when a positroid is a sparse paving matroid, and we express this characterization in terms of various combinatorial objects that parameterize positroids.

Throughout, positroids have an ordered ground set.
This is the usual convention when considering the Grassmannian and total positivity: the ordering corresponds to the order of columns in a matrix representing a linear subspace.
With this convention, many combinatorial objects are in bijection with positroids.
The objects we work with are Grassmann necklaces, decorated permutations, and $\Le$-diagrams, each defined by Postnikov~\cite{Pos}.
These objects can be used to parameterize cells in a decomposition of the totally nonnegative Grassmannian.

There is a growing body of work on matroid-theoretic aspects of positroids.
This includes the work of Park~\cite{Park} on finding which paving matroids are excluded minors for a subclass of positroids known as interval positroids.
Some of these works consider positroids on unordered ground sets~\cite{BoninPos}, and others explicitly distinguish between ordered and unordered ground sets~\cite{Quail}.
This amounts to considering positroids up to matroid isomorphism.
In our setting, there can be multiple distinct positroids that are isomorphic as matroids on an unordered ground set.
Whether a positroid is sparse paving depends only on the matroid isomorphism class.
However, we work with an ordered ground set to develop methods for testing whether a positroid is sparse paving that can be applied directly to any cell in the totally nonnegative Grassmannian.
Thus, our enumeration results also count positroids on ordered ground sets.

In Section~\ref{sec:prelim}, we provide background on matroid theory, including definitions of paving and sparse paving matroids.
In Sections~\ref{sec:necklace}--\ref{sec:Le}, we state and prove our characterization of sparse paving positroids.
Each of these sections defines and focuses on a different combinatorial model for positroids.
Finally, in Section~\ref{sec:enum}, we count the number of sparse paving positroids.

\section{Preliminaries of matroids}
\label{sec:prelim}

We provide a brief review of matroids and relevant terminology. 
For more details about matroid theory, one may see~\cite{Oxley} and~\cite{Welsh}.
We begin by defining a matroid.
Let $E$ be a finite set, referred to as the \emph{ground set}.
A \emph{matroid basis system on ground set $E$} is a nonempty family $\B$ of subsets of $E$ satisfying the \emph{exchange axiom}: for all distinct $B,B'\in\B$ and all $e\in B\sm B'$, there exists $e'\in B'\sm B$ such that $\left(B\sm \{e\}\right)\cup\{e'\}\in \B$.
Elements of $\B$ are called \emph{bases}.
The pair $M=(E,\B)$ defines a \textit{matroid}.
Any subset of a basis is called an \textit{independent set}, and a \emph{dependent set} is a subset of $E$ which is not independent.
The \emph{rank function} of $M$ is defined by $\rank(A)=\max\{|A\cap B|:\ B\in\B\}$ for $A\subseteq E$.
The number $\rank(E)$ is called the \textit{rank} of the matroid $M$ and can also be denoted $\rank(M)$.

Some additional matroid terminology that will be useful:
\begin{itemize}
\item A \textit{circuit} in a matroid is a minimal dependent subset of $E$. A \textit{loop} is a circuit of size~1.
\item A \textit{flat} of $M$ is a subset $F\subseteq E$ such that $\rank(G)>\rank(F)$ for every $G\supsetneq F$.
\item A \textit{hyperplane} is a flat of rank $\rank(M)-1$. 
\item The \textit{dual} of a matroid $M=(E,\B)$ is the matroid $M^*:=(E,\B^*)$ where 
\[\B^*:=\{E\setminus B: B\in \B\}.\]
\end{itemize}

In general, proving a collection of sets is a matroid base system is not a trivial task, though some examples require less work than others. 
For example, for a nonnegative integer $k$ and ground set $E$, the \textit{uniform matroid} $U_{k}(E)$ has basis system $\binom{E}{k} := \{B\subseteq [n]:\ |B|=k\}$.
We leave it to the reader to verify that this collection of sets indeed satisfies the exchange axiom.
We will often find it convenient to write $U_{k,n}:=U_k([n])$, where $[n]:=\{1,2,\dots, n\}$.

We will be especially interested in a class of matroids called \emph{sparse paving matroids}.
A matroid of rank~$k$ is called \emph{paving} if every circuit has cardinality greater than or equal to~$k$, and it is \emph{sparse paving} if it and its dual are both paving.
It is conjectured and widely believed that asymptotically all matroids are paving matroids---or even sparse paving matroids---see \cite{paving}, \cite{sparsePaving}, \cite[Chapter~15.5]{Oxley}.

There are several alternative ways of defining sparse paving matroids. Given a rank $k$ matroid $M=(E,\mathcal{B})$, let $\mathcal{CH}={E\choose k}\setminus \mathcal{B}$. Then $M$ is sparse paving if any (and hence all) of the following hold: 
\begin{enumerate}
    \item $\mathcal{CH}$ is the set of circuit-hyperplanes---that is, sets which are both circuits and hyperplanes---for $M$.
    \item Given $C,C'\in \mathcal{CH}$, $|C\triangle C'|\geq 4$, where $\triangle$ is the \textit{symmetric difference} operation.
\end{enumerate}
Circuit-hyperplanes are involved in an important operation in matroid theory.
\begin{lemma}{\cite[Prop.1.5.14]{Oxley}}
    Let $C$ be a \textit{circuit-hyperplane}, that is, a set that is both a circuit and a hyperplane.  If $C$ is a circuit-hyperplane, then the family $\B\cup\{C\}$ is in fact a matroid basis system, called the \textit{relaxation of $M$ at $C$}. 
\end{lemma}
In particular, this leads to one more alternative definition of sparse paving matroids in addition to the two given above: 
\begin{enumerate}
    \item[(3)] After relaxing all circuit-hyperplanes for $M$, the matroid yields a uniform matroid. 
\end{enumerate}

Recently, the notion of relaxation was generalized from circuit-hyperplanes to something called ``stressed-subsets" in \cite{stress_sub}. This operation allows one to define a class of matroids called \textit{elementary split matroids} in an analogous way to how the original relaxation operation corresponds to sparse paving matroids. Elementary split matroids---which contain the class of sparse paving matroids---are outside the scope of this paper, so we include this here just to provide an updated account of the relevant theory.

\section{Grassmann Necklaces of Sparse Paving Positroids}
\label{sec:necklace}
Take a positive integer $n$ and $t \in [n]$.
We define an ordering denoted $<_t$ on $[n]$ by
\[t <_t t+1 <_t \cdots <_t n <_t 1 <_t 2 <_t \cdots <_t t-1\]

For any $k$ and any two $k$-element subsets $I = \{a_1 <_t a_2 <_t \cdots <_t a_k\}$ and $J = \{b_1 <_t b_2 <_t \cdots <_t b_k\}$ of $[n]$ we say that $I \leq_t J$ if $a_i \leq_t b_i$ for all $1 \leq i \leq k$.

\begin{definition}[{\cite[Definition 16.1]{Pos}}]
    A \emph{Grassmann necklace} is a sequence $(I_1, I_2, \cdots, I_n)$ of subsets of $[n]$ so that:
    \begin{itemize}
        \item if $i \in I_i$, then $I_{i+1} = (I_i \setminus \{i\}) \cup \{j\}$ for some $j \in [n]$,
        \item and if $i \not \in I_i$, then $I_{i+1} = I_i$.
    \end{itemize}
    Here, all indices are considered modulo $n$.
\end{definition}

For any $k$-element subset $I \in \binom{[n]}{k}$ and $t \in [n]$ we define the \emph{cyclically shifted Schubert matroid} $SM^t_I$ as the matroid with bases
\[\left\{J \in \binom{[n]}{k} : I \leq_t J\right\}.\]
We now recall a result of Oh~\cite[Theorem 8]{Oh} that allows us to define positroids as intersections of shifted Schubert matroids.
Given a Grassmann necklace $(I_1, I_2, \cdots, I_n)$, the positroid corresponding to it is
\[\bigcap_{t=1}^n SM^t_{I_t}\]
where the intersection is taken as a set of bases.

For any $1 \leq k \leq n$ and $i \in [n]$, let us define the subset $C_{k,n}^{(i)} \in \binom{[n]}{k}$ to be $C_{k,n}^{(i)} = \{i, i+1, \dots i + k -1\}$ where everything is computed modulo $n$ using $[n]$ as the set of representatives.
So, $C_{k,n}^{(i)}$ is simply the cyclic interval in $[n]$ starting at $i$ with $|C_{k,n}^{(i)}| = k$.
These subsets play a key role in our classification of Grassmann necklaces for sparse paving positroids since $C_{k,n}^{(i)}$ is the minimal element in $\binom{[n]}{k}$ for the order $<_i$.

\begin{theorem}
    The positroid on ground set $[n]$ of rank $2 \leq k \leq n-2$ corresponding to the Grassmann necklace $(I_1, I_2, \dots, I_n)$ is sparse paving if and only if for every $i \in [n]$, $I_i \neq C_{k,n}^{(i)}$ implies that $I_{i-1} = C_{k,n}^{(i-1)}$, $I_{i+1} = C_{k,n}^{(i+1)}$, and $I_i = (C_{k,n}^{i} \setminus \{i + k - 1\}) \cup \{i+k\}$.
    Moreover, in this case, its set of circuit hyperplanes consists precisely of the $C_{k,n}^{(i)}$ for which $I_i \neq C_{k,n}^{(i)}$.
    \label{thm:necklace}
\end{theorem}
\begin{proof}
    Consider a sparse paving positroid of rank $k$ with Grassmann necklace $(I_1, I_2, \dots, I_n)$.
    Recall, this means that $I_i$ is the lexicographically least basis with the order $<_i$.
    If there exists an $i$ so that $I_{i} \neq C_{k,n}^{(i)}$ and $I_{i+1} \neq C_{k,n}^{(i+1)}$, then neither $C_{k,n}^{(i)}$ nor $C_{k,n}^{(i+1)}$ are bases.
    Such a positroid cannot be sparse paving because $|C_{k,n}^{(i)} \Delta C_{k,n}^{(i+1)}| = 2$.
    If there exists an $i$ so that $I_{i} \neq C_{k,n}^{(i)}$ and $I_i \neq C_{k,n}^{i} \setminus \{i + k -1\} \cup \{i+k\}$, then again the positroid is not sparse paving because these two non-bases have symmetric difference equal to $2$.
    Thus one direction of the theorem is proven.

     To prove the other direction take a Grassmann necklace $(I_1, I_2, \dots, I_n)$ satisfying the conditions of the theorem.
     The only $k$-subsets that are not bases are $C_{k,n}^{(i)}$ for $i \in [n]$ with $I_i \neq C_{k,n}^{(i)}$.
    This is because we know that our positroid is an intersection of shifted Schubert matroids, and with respect to $<_i$ the smallest $k$-subset is $C_{k,n}^{(i)}$ while the second smallest $k$-subset is $(C_{k,n}^{i} \setminus \{i + k -1\}) \cup \{i+k\}$.
    Since $|C_{k,n}^{(i)} \Delta C_{k,n}^{(j)}| \geq 4$ whenever $j \neq i \pm 1$ and $2 \leq k \leq n-2$, the positroid under consideration is sparse paving, and the theorem is proven.
    \end{proof}

    We say a set $A\subseteq[n]$ is \emph{non-adjacent} if whenever $i\in A$ we have $i-1\notin A$ and $i+1\notin A$, where this condition is considered modulo $n$. Thus, 
    Theorem~\ref{thm:necklace} tells us the following. 

    \begin{corollary}
Sparse paving positroids on the ground set $[n]$of rank $2 \leq k \leq n-2$ are in bijection with non-adjacent subsets of $[n]$.
\label{cor:nonadj}
    \end{corollary}

    \begin{proof}
        Send a Grassmann necklace $(I_1, I_2, \dots, I_n)$ of a sparse paving positroid to the non-adjacent subset
    $\{i \in [n] : I_i \neq C_{k,n}^{(i)}\}$.
    \end{proof}

\section{Decorated Permutations of Sparse Paving Positroids}
\label{sec:decperm}

\begin{definition}[{Decorated Permutation~\cite[Definition 13.3]{Pos}}]
A \emph{decorated permutation} is a pair $(\pi, \phi)$ where $\pi$ is a permutation of $[n]$ and $\phi$ is a function from the set of fixed points $\{i \in [n] : \pi(i)=i\}$ to $\{-1,1\}$.
\end{definition}

Postnikov~\cite[Section 16]{Pos} has described how to obtain the decorated permutation corresponding to the same positroid from a Grassmann necklace.
Given a Grassmann necklace $(I_1, I_2, \dots, I_n)$ which defines a decorated permuation $(\pi, \phi)$ by
\begin{itemize}
    \item $\pi(i)=j$ if $I_{i+1} = (I_i \setminus \{i\}) \cup \{j\}$ and $j \neq i$,
    \item $\pi(i) = i$ and $\phi(i)=1$ if $I_i = I_{i+1}$ and $i \not\in I_i$,
    \item $\pi(i) = i$ and $\phi(i)=-1$ if $I_i = I_{i+1}$ and $i \in I_i$.
\end{itemize}
When $I_i = C_{n,k}^{(i)}$ for $1 \leq i \leq n$, then $\pi(i) = i+k$ for $1 \leq i \leq n$.
This corresponds to the uniform matroid $U_{k,n}$, and we will denote this decorated permutation by $\pi_{k,n}^{top}$.
For a particular example with $n = 6$ and $k=3$ we have $\pi = 456123$ in one-line notation.
We will explain that all decorated permutations of sparse paving positroids can be obtained from the decorated permutation of the uniform matroid by transposing adjacent elements.
As usual, the condition of adjacency is viewed cyclically, with the first and last elements of the permutation considered to be adjacent.
Sticking with $n=6$ and $k=3$ assume that $I_i = C_{k,n}^{(i)}$ for $i \neq 3$ with 
\[I_3 = (C^{(3)}_{3,6} \setminus \{5\}) \cup \{6\} = \{3,4,6\}\]
so we then obtain $\pi = 465123$.
Here we have swapped the second and third positions of $456123$ to obtain $465123$.

Define $\sigma_i$ to operate on permutations in one-line by exchanging their entries in positions $i$ and $i-1$.
Here $\sigma_1$ exchanges the first and last positions.
For any $A \subseteq [n]$ that is non-adjacent define
\[\sigma_A := \prod_{i \in A} \sigma_i\]
which is well-defined because the non-adjacent condition means the operators in the product commute.

\begin{theorem}
    For positive integers $n$ and $2 \leq k \leq n-2$, a decorated permutation $(\pi, \phi)$ corresponds to a sparse paving positroid if and only if $\pi = \sigma_A (\pi_{k,n}^{top})$ for some non-adjacent $A \subseteq [n]$.
    In this case, $\pi$ has no fixed points and so the function $\phi$ gives no information.
\end{theorem}
\begin{proof}
Take a Grassmann necklace $(I_1, I_2, \dots, I_n)$ satisfying the conditions in Theorem~\ref{thm:necklace} so that it is in bijection with the non-adjacent subset $A \subseteq [n]$ as in Corollary~\ref{cor:nonadj}.
Let $(\pi, \phi)$ be the decorated permutation associated to the same positroid as the Grassmann necklace.
Assume that $i \in A$ so that $I_{i-1} = C_{k,n}^{(i-1)}$, $I_i = \left(C_{k,n}^{(i)} \setminus \{i+k-1\}\right) \cup \{i+k\}$, and $C_{i+1} = C_{k,n}^{(i+1)}$.
This implies that $\pi(i-1) = i + k$ and $\pi(i) = i+k-1$.
The case that $i \not\in A$ and $i+1 \in A$ can be handled in the same way by shifting indices.

Now, assume $i \not\in A$ and $i+1 \not\in A$ so that $I_i = C_{k,n}^{(i)}$ and $I_{i+1} = C_{k,n}^{(i+1)}$.
This implies $\pi(i) = i+k$.
The equality $\pi = \sigma_A (\pi_{k,n}^{top})$ then follows.
Furthermore, since $2 \leq k \leq n-2$, we find that $\pi$ has no fixed points.
\end{proof}

\section{$\Le$-diagrams of Sparse Paving Positroids }
\label{sec:Le}
\begin{definition}[{$\Le$-diagram~\cite[Definition 6.1]{Pos}}]
For any partition $\lambda$, we consider its Young diagram in the convention that our diagrams are left justified with the largest part on top.
A \emph{$\Le$-diagram} $D$ is a filling of the Young diagram so that each box is either empty or contains a $\bullet$ such that for $i < i'$ and $j < j'$ when the box $(i',j')$ exists it contains a $\bullet$ whenever both the boxes $(i,j')$ and $(i',j)$ contain a $\bullet$.
We let $|D|$ denote the number of $\bullet$'s that $D$ contains, and we let $\Le(k,n)$ denote the collection of all $\Le$-diagrams that fit inside a $k \times (n-k)$ box.
\end{definition}

Given $D \in \Le(k,n)$, the southeast boundary of the shape, along with possibly the top and left sides of the box, gives us a lattice path from the top right to the bottom left of the box.
We will start in the upper right corner of the $k \times (n-k)$ box and label each step with an element from $[n]$ in order.
Vertical steps are sources, and horizontal steps are sinks.
Any $\bullet$ in $D$ is an internal vertex.
A directed edge is drawn from right to left from each source or internal vertex whenever there is an internal vertex to the right in the same row.
In addition, a directed edge is drawn from an internal vertex whenever there is an internal vertex or sink below in the same column.
We will denote this directed network obtained from $D$ by $N_D$.

A path $p$ in $N_D$ is a sequence of vertices $v_0, v_1, \dots, v_m$ so that $v_a \to v_{a+1}$ is a directed edge for each $0 \leq a < m$.
Here we say $p$ \emph{terminates} at $v_m$ and write $t(p) = v_m$.
Let $N_D$ have $I$ and $J$ of its set of sources and sinks, respectively.
So, $|I| = k$, $I \cup J = [n]$, and $I \cap J = \emptyset$.
A \emph{disjoint path system} is a collection $(p_i)_{i \in I}$ of directed paths in $N_D$ with $p_i$ starting at source $i$ such that:
\begin{enumerate}
    \item $t(p_i) \in j \in J$ is a sink or $t(p_i) = i$,
    \item and $p_{i}$ and $p_{i'}$ do not have any vertex in common for any $i, i' \in I$ with $i \neq i'$.
\end{enumerate}
We say the disjoint path system $(p_i)_{i \in I}$ \emph{realizes} the set $\{t(p_i) : i \in I\} \subseteq [n]$.
For any $D \in \Le(k,n)$, the positroid associated to $D$ is the matroid on $[n]$ that has a basis for each set that is realized by a disjoint path system (see \cite[Proposition 13]{Oh} and \cite[Theorem 6.5]{Pos}).

In this section, we will determine when a $\Le$-diagram corresponds to a sparse paving matroid. 
First, recall that the uniform matroid $U_{k,n}$ is a positroid and its $\Le$-diagram is the full $k \times (n-k)$ rectangle with a $\bullet$ in each possible position.
To begin we number $n$ cells along the boundary of the $k \times (n-k)$ rectangle. If $2 \leq k \leq n-1$, we do the following:
\begin{enumerate}
    \item[(i)] We number the cell in the lower right corner with a $1$.
    \item[(ii)] We number, from right to left, the cells in the top row with $2, 3, \dots n - k+1$.
    \item[(iii)] We number, from top to bottom, the cells in the left-most column with $n-k+1, n-k+2, \dots, n$.
\end{enumerate}









\begin{figure}[h]
     \begin{center}

         \begin{ytableau}
          7& 6 & 5 & 4 & 3 & 2\\
 8& & & & &\\
 9& & & & & \\
 10&& & & & 1
 \end{ytableau}

     \end{center}

     \caption{The numbering of the cells for $k = 4$ and $n = 10$.}
     \label{fig:le_numbering}
 \end{figure}

Given a $k \times (n-k)$ rectangle that is completely filled with $\bullet$'s, a numbered cell is the only place a $\bullet$ can be removed if we wish to still have $\Le$-diagram.
For $A\subseteq [n]$ we define $D_{k,n}(A) \in \Le(k,n)$ to be the $\Le$-diagram obtained from the completely filled $k \times (n-k)$ rectangle by removing the $\bullet$ in the cell numbered $i$ for each $i \in A$.
When $i=1$, we modify the diagram by removing the $\bullet$ and the cell in that position.
See Figure~\ref{fig:le_numbering} for an example of the cell numbering and Figure~\ref{fig:le_sparse_pave} for examples of the construction.

\begin{figure}[h]
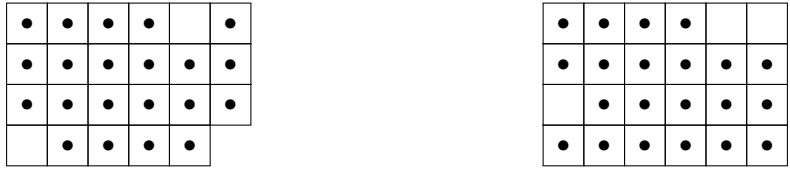

\begin{minipage}{.45\textwidth}
    \begin{center}
        \begin{ytableau}
\bullet& \bullet  & \bullet & \bullet &  & \bullet\\
\bullet&\bullet &\bullet &\bullet &\bullet &\bullet\\
\bullet& \bullet &\bullet &\bullet &\bullet &\bullet \\
&\bullet&\bullet&\bullet &\bullet  
\end{ytableau}
    \end{center}
\end{minipage}
\begin{minipage}{.45\textwidth}
    \begin{center}
        \begin{ytableau}
\bullet& \bullet  & \bullet & \bullet &  & \\
\bullet&\bullet &\bullet &\bullet &\bullet &\bullet\\
 & \bullet &\bullet &\bullet &\bullet &\bullet \\
\bullet&\bullet&\bullet&\bullet &\bullet  & \bullet
\end{ytableau}
    \end{center}
\end{minipage}

    \caption{On the left, we have the $\Le$-diagram $D_{4,10}(A)$, where $A=\{1,3,10\}$.
    On the right, we have $D_{4,10}(A)$ where $A=\{2,3,9\}$. }
    \label{fig:le_sparse_pave}
\end{figure}

\begin{proposition}
$D_{k,n}(A)$ is a $\Le$-diagram.
\end{proposition}
\begin{proof}
    By construction, $D_{k,n}(A)$ is a $\Le$-diagram because of where we are allowed to remove $\bullet$'s. When $1 \in A$ we remain a $\Le$-diagram because we also remove the lower right cell.
\end{proof}

\begin{lemma}
For any $i \in [n]$, the positroid associated to $D_{k,n}(A)$ has $C_{k,n}^{(i)}$ as a basis if and only if $i \not \in A$.
\label{lem:missing_basis}
\end{lemma}
\begin{proof}
First, consider $i=1$.
If $1 \not\in A$, then $C_{k,n}^{(1)}$ is a basis as it is the set of sources.
If $1 \in A$, then $C_{k,n}^{(1)}$ is not a basis as the set of sources is $[k-1] \cup \{k+1\}$ and thus $C_{k,n}^{(1)} = [k]$ cannot be realized with a disjoint path system as it would require a path from the source $k+1$ to $k$ which cannot exist.

Next, consider $2 \leq i \leq n-k+1$ where we have $C_{k,n}^{(i)} = \{i, i+1, \dots, i+k-1\}$.
In particular $1 \not \in C_{k,n}^{(i)}$.
So, any disjoint path system realizing $C_{k,n}^{(i)}$ as a basis would require a path from the source $1$ to the leftmost sink $i+k-1$, the largest number in $C_{k,n}^{(i)}$. In the case where $i\in A$, we claim no disjoint path system witnessing $C_{k,n}^{(i)}$ can exist. In this case, the cell in row $1$ above sink $i+k-1$ is empty. 
Now, let $s_1 < s_2 < \cdots < s_m$ be the sources\footnote{We know that $s_1 = 1$. Also, when $i \leq k$ we have $m = i-1$ and $s_m = i-1$. When $i$ becomes larger, if $1 \not\in A$ we will have $s_m=k$ because in this case $k$ is a source. Otherwise, $s_m = k+1$ when $1 \in A$.} not in $ C_{k,n}^{(i)}$.
Also, let $t_1 < t_2 < \cdots < t_m$ be the sinks\footnote{For $j \geq 2$, $t_j = k+j$, while $t_1 = k$ or $t_1 = k+1$ depending whether $1 \in A$ or not.} in $C_{k,n}^{(i)}$.
So, there would need to be disjoint paths from sources $s_1, s_2, \dots, s_m$  to sinks $k+1,k+2, \dots, k+m$.
This is only possible if the path from $s_m$ to $t_1$ turns above $t_1$, the path from $s_{m-1}$ to $t_2$ turns above $t_2$, and so on, until the path from $1=s_1$ to $t_m$ turns above $t_m=i+k-1$. 
See Figure \ref{fig:pldc_proof}.

If $i \not \in A$, then the aforementioned path from $1$ to $i+k-1$ exists since the cell in row $1$ above sink $i+k-1$ has a $\bullet$ in it.
Furthermore, by the definition of $D_{n,k}(A)$, nothing has been modified in the rows of any source other than $1$.
Thus, a disjoint path system can exist, making $C_{k,n}^{(i)}$ a basis.

Lastly, consider $n-k+2 \leq i \leq n$.
In this case $C_{k,n}^{(i)} = \{i, i+1, \cdots, n\} \cup \{1, 2, \cdots, k - n + i - 1\}$.
This means that in any disjoint path system that realizes $C_{k,n}^{(i)}$ as a basis, the smallest labeled source not in $C_{k,n}^{(i)}$ must be connected to the sink $n$.
Here, we can argue the same as previously, with the source $k-n+i$ playing the same role the source $1$ did earlier.
So, if $i \in A$, then this path is impossible and $C_{k,n}^{(i)}$ is not a basis.
However, if $i \not\in A$, such a path is possible and $C_{k,n}^{(i)}$ is a basis, as all other paths needed for the disjoint path system use portions of $D_{k,n}(A)$ that have not been modified from the completely filled case.
\end{proof}

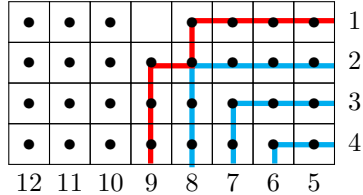
\begin{figure}[h]
\begin{tikzpicture}
\draw[line width=.2em,cyan] (1.9,0.5)--(0,0.5)--(0,-.85);
\draw[line width=.2em,cyan] (1.9,0.0)--(.55,0.0)--(.55,-.8);
\draw[line width=.2em,cyan] (1.9,-0.55)--(1.1,-0.55)--(1.1,-.8);
\draw[line width=.2em,red] (1.9,1.1)--(0,1.1)--(0,0.5)--(-.55,0.5)--(-.55,-.8);
\node{\begin{ytableau}
\bullet& \bullet& \bullet &   & \bullet & \bullet  & \bullet & \bullet &\none[1]\\
\bullet & \bullet& \bullet&\bullet &\bullet &\bullet &\bullet &\bullet&\none[2]\\
\bullet& \bullet& \bullet & \bullet & \bullet &\bullet &\bullet &\bullet &\none[3]\\
\bullet &\bullet  & \bullet  &\bullet  &\bullet&\bullet &\bullet &\bullet &\none[4]\\
\none[12] & \none[11] & \none[10] & \none[9] & \none[8]&\none[7] & \none[6] & \none[5]
\end{ytableau}};
\end{tikzpicture}
    \caption{The $\Le$-diagram $D_{4,12}(\{6\})$ and an illustration that $C_{4,12}^{(6)}=\{6,7,8,9\}$ is not a basis due to the path from $1$ to $9$ colliding with another path above sink $8$.} \label{fig:pldc_proof}
\end{figure}

\begin{theorem}
    For $2 \leq k \leq n-2$, the positroid associated to the $\Le$-diagram $D_{k,n}(A)$ is sparse paving if and only if $A$ is non-adjacent.
        Moreover, the sparse paving positroid obtained has circuit-hyperplanes $\{C_{k,n}^{(i)} : i \in A\}$.
\end{theorem}
\begin{proof}
    Given the positroid associated to $D_{k,n}(A)$ for some $A$, by Lemma~\ref{lem:missing_basis} we know that $C_{k,n}^{(i)}$ is not a basis if and only if $i \in A$.
    Also, by Theorem~\ref{thm:necklace} and Corollary~\ref {cor:nonadj}, we know the sparse paving positroids are exactly characterized by having their set of circuit hyperplanes consist of exactly $C_{k,n}^{(i)}$ for $i \in A$ such that $A \subseteq [n]$ is non-adjacent.
    Removing $\bullet$'s from a $\Le$-diagram can only decrease the number of bases of the corresponding positroid.
    The theorem then follows because the $\Le$-diagrams described in this theorem are in bijection with non-adjacent $A \subseteq [n]$.
\end{proof}

\section{Counting Sparse Paving Positroids}
\label{sec:enum}

Counting sparse paving positroids of rank $k=1$, or dually $k = n-1$, is trivial.
In this case, we have $n+1$ sparse paving positroids.
A sparse paving positroid with $k=1$ is either the uniform matroid $U_{1,n}$ or a rank $1$ matroid where exactly one element of $[n]$ is chosen to be a loop.

By Corollary~\ref{cor:nonadj}, we know that sparse paving positroids on $[n]$ for any rank $2 \leq k \leq n-2$ are in bijection with non-adjacent subsets of $[n]$.
Let $s_n$ denote the number of non-adjacent subsets of $[n]$ for each $n \geq 0$.
The number of sparse paving positroids is independent of the rank $k$ whenever $2 \leq k \leq n-2$, in which case we need $n \geq 4$.

\begin{proposition}
The number of non-adjacent subsets of $[n]$ satisfies the recurrence $s_n = s_{n-1} + s_{n-2}$ for $n \geq 4$ with $s_0 = 1$, $s_1 = 2$, $s_2 = 3$, and $s_3 = 4$.
    \label{prop:sn}
\end{proposition}
\begin{proof}
The only non-adjacent subset for $n=0$ is $\{\}$.
The non-adjacent subsets for $n=1$ are $\{\}$ and $\{1\}$.
When $n=2$ the non-adjacent subsets are $\{\}$, $\{1\}$, and $\{2\}$.
For $n=3$ the non-adjacent subsets are $\{\}$, $\{1\}$, $\{2\}$, and $\{3\}$.
This handles the initial conditions.

Now assume $n \geq 4$. Consider an arbitrary non-adjacent subset $A \subseteq [n]$. There are three cases. 
\begin{enumerate}
    \item Suppose $n$ is not in $A$ and at most one of $1$ or $n-1$ is in $A$. These sets are in bijection with non-adjacent subsets of $[n-1]$.
    \item Suppose both $1$ and $n-1$ are in $A$ (and thus $n$ is not in $A$). These are in bijection with non-adjacent subsets of $[n-2]$ which contain $1$, achieved by adding the element $n-1$ to the set. Since $n \geq 4$, we know that $n-1 \geq 3$ and is not adjacent to $1$.
    \item Suppose $n$ is in $A$ (and thus neither $n-1$ nor $1$ is in $A$). These are in bijection with non-adjacent subsets of $[n-2]$ which do \textit{not} contain $1$, achieved by adding the element $n$ to the set.
\end{enumerate}

These three cases are disjoint and make up every possible non-adjacent subset of $[n]$. Moreover, case (1) is counted by $s_{n-1}$, while cases (2) and (3) together are counted by $s_{n-2}$. 
\end{proof}

One can find the sequence $(s_n)_{n \geq 0}$ in the Online Encyclopedia of Integer Sequences~\cite [A169985]{oeis}.
From this, we see that we can also compute $s_n$ as the nearest integer to $\phi^n$, where $\phi=\frac{1+\sqrt{5}}{2}$ is the golden ratio. 
If we consider the \emph{Lucas numbers} defined by $L_0 = 2$, $L_1 = 1$, and $L_n = L_{n-1} + L_{n-2}$ for $n\geq 2$, then we have that $s_n = L_n$ for $n \geq 2$.

\bibliographystyle{alphaurl}
\bibliography{sample}

\end{document}